\documentclass[reqno]{amsart}
\usepackage[utf8]{inputenc}
\usepackage{amsmath, amssymb, amsthm, tikz} 
\usepackage{enumerate}
\usepackage{bbm}

\usepackage[charter]{mathdesign}

\usepackage[colorlinks=true, pdfstartview=FitV, linkcolor=blue, citecolor=blue, urlcolor=blue,pagebackref=false]{hyperref}
\hypersetup{ colorlinks, citecolor=black, filecolor=black, linkcolor=blue, urlcolor=black }

\newtheorem{theorem}{Theorem}[section]

\newtheorem{proposition}[theorem]{Proposition}
\newtheorem{lemma}[theorem]{Lemma}

\numberwithin{equation}{section}

\theoremstyle{definition}

\theoremstyle{remark}
\newtheorem{remark}[theorem]{Remark}

\newcommand{\dd}{\,\mathrm{d}}
\newcommand{\PP}{\mathbb{P}}

\newcommand{\EE}{\mathbb{E}}
\newcommand{\RR}{\mathbb{R}}
\newcommand{\NN}{\mathbb{N}}
\newcommand{\ZZ}{\mathbb{Z}}

\newcommand{\lvec}[1]{\overleftarrow{#1}}
\renewcommand{\vec}[1]{\overrightarrow{#1}}

\newcommand{\mc}[1]{\mathcal #1}

\newcommand{\ind}[1]{\mathbbm{1}_{#1}}
\newcommand{\Ind}[1]{\mathbbm{1}_{\{{#1}\}}}
\newcommand{\lrp}[1]{\left(#1\right)}
\newcommand{\lrc}[1]{\left[#1\right]}
\newcommand{\lrch}[1]{\left\{ #1\right\} }
\newcommand{\lrb}[1]{\left\langle #1 \right\rangle}

\newcommand{\bb}[1]{{\mathbbm #1}}

\newcommand{\para}{\lambda}
\newcommand{\ann}{Q_\mu}

\begin{document}

\title{Variable speed symmetric random walk driven by symmetric exclusion}

\author{Ot\'{a}vio Menezes}
\address{Ot\'{a}vio Menezes, Department of Mathematics, Purdue University, West Lafayette, USA}
\email{omenezes@purdue.edu}
\author{Jonathon Peterson} 
\address{Jonathon Peterson, Department of Mathematics, Purdue University, West Lafayette, USA}
\email{peterson@purdue.edu}
\author{Yongjia Xie}
\address{Yongjia Xie, Department of Mathematics, Purdue University, West Lafayette, USA}
\email{xie287@purdue.edu}

\maketitle

\begin{abstract}
We prove a quenched functional central limit theorem for a one-dimensional random walk driven by a simple symmetric exclusion process. This model can be viewed as a special case of the random walk in a balanced random environment, for which the weak quenched limit is constructed as a function of the invariant measure of the environment viewed from the walk. We bypass the need to show the existence of this invariant measure. Instead, we find the limit of the quadratic variation of the walk and give an explicit formula for it.
\end{abstract}

\section{Introduction}

We prove a quenched functional central limit theorem for a one-dimensional random walk driven by a simple symmetric exclusion process. The model belongs to the class of random walks in dynamical random environments. 
Recent works have studied examples where the environment is an interacting particle system, including independent random walks \cite{rwonrw}, the contact process \cite{contact} and the simple symmetric exclusion process (SSEP). 

To define a random walk driven by the SSEP, one fixes parameters $p_1, p_0, \rho \in [0,1]$,  $\lambda_0, \lambda_1 >0$ and makes the random walk jump from $x\in \bb Z$ to $x+1$ at time $t$ at rate $\lambda_1 p_1\eta_t(x) + \lambda_0 p_0(1-\eta_t(x))$, where $\eta_t(x)$ is the state of the exclusion process (either $0$ or $1$) at site $x$ and time $t$, started from equilibrium at density $\rho$. The rate for a jump from $x$ to $x-1$ is $\lambda_1(1-p_1)\eta_t(x)+\lambda_0(1-p_0)(1-\eta_t(x))$. 
Several cases were studied. The results in \cite{hs} and \cite{hkt} that we are about to cite were proven for a discrete-time random walk, but we believe that the continuous-time results we state are true as well.
In \cite{hs}, laws of large numbers and Gaussian fluctuations are proven for $\lambda_0 = \lambda_1$ sufficiently large or sufficiently small and appropriate assumptions on $p_0$ and $p_1$.
When $\lambda_0 = \lambda_1 $, \cite{dossantos} proves that the limiting speed, if any, is strictly between $\lambda_0(2p_0-1)$ and $\lambda_1(2p_1-1)$. In \cite{hkt} it is proven that, for $\lambda_0=\lambda_1 =1$ the law of large numbers holds for all $\rho$, with only two possible exceptions. When the speed is not zero a Gaussian central limit theorem holds and when $p_0 = 1-p_1$ (as in \cite{adhrldp} and \cite{hs}) and $\rho = 1/2$ the speed is zero.      
It is an interesting open problem to find the scale of the fluctuations when $\lambda_0 = \lambda_1$ and $p_0 = 1-p_1$. In this case, only the law of large numbers is known (\cite{hkt}). It is conjectured in \cite{jm}, where the central limit theorem for a weakly asymmetric version of the model was considered (see also \cite{afjv} and \cite{ajv}), that the fluctuations are of order $t^{3/4}$.
A related model where space is continuous was introduced in \cite{hs2}.

Here we allow $\lambda_0 \neq \lambda_1$ but assume $p_0 = p_1 = \frac{1}{2}$. In this setting, the random walk is a time-change of a simple symmetric random walk. The law of large numbers is immediate, and the problem is to prove convergence to Brownian motion and compute the variance of this limiting Brownian motion at time $t$. We perform this computation when the environment starts in equilibrium at density $\rho \in [0,1]$. With those assumptions, our model falls into the class of balanced dynamic random environments. For this class of models an invariance principle was proved in  \cite{guoramirez}. 
In this paper we give an entirely different proof of the invariance principle for this particular model. Since random walks in balanced environments are martingales, the key to proving an invariance principle is in proving that the quadratic variation grows linearly. In all previous proofs of invariance principles for random walks in (static or dynamic) environments this was accomplished by proving the existence of an invariant measure for the environment viewed from the particle that was absolutely continuous with respect to the initial measure on environments (see e.g., \cite{law82,gz12,bd14,guoramirez}). 
In this paper, however, we are able to prove the linear growth of the quadratic variation without any reference to the existence of invariant measures for the environment viewed from the particle. Not only does this give a simpler proof of the invariance principle for this particular model, but it also enables us to compute explicitly the scaling constant in the invariance principle and allows us to obtain quantitative estimates on the rate of convergence for the quadratic variation, see \eqref{annealed-ub}. 

Since the underlying dynamic environment in our model has only two types of sites (particles/holes), the key to analyzing the growth rate of the quadratic variation is to compute the asymptotic fraction of time , $\lim_{t\to \infty}t^{-1}\int_0^t \eta_{X_s}(s)\dd s$. We accomplish this by providing an explicit function $\varphi$ and explicit constants $a$ and $b$ such that $L\varphi \approx a\xi_0 + b$, where $\xi_x(t):=\eta_t(x+X_t)$ and $L$ denotes the generator of the process $(\xi(t))_{t\geq 0}$, the environment as seen by the walk. 
This technique of estimating additive functionals $\int_0^t g(\xi(s))\dd s$ by solving the equation $g(\xi) \approx a + b\, u(\xi)$ was introduced in \cite{kipnisvaradhan}. In the context of random walks in random environments, it has been used in \cite{avenassep}, \cite{kozmatoth} and \cite{agha}, among other works.  


\section{Model and statement of the theorem}
 
 Let $\rho, \lambda \in [0,1]$ and $T>0$ be fixed throughout the paper. Denote by $\mu=\bigotimes_{x\in\ZZ} \text{Ber}(\rho)$ the probability measure on $ \{0,1\}^{\bb Z} $ under which the random variables $\{\eta_x\}_{x\in \bb Z}$ are i.i.d. of mean $\rho$. We consider a nearest-neighbour random walk on $ \bb Z $, driven by the simple symmetric exclusion process (SSEP) with initial distribution $\mu$. Define the joint law of the random walk and the SSEP by the Markov generator
\begin{equation}\label{generator of the joint process}
\begin{aligned}
L^{\mathrm{joint}}f(\eta, x) &  = \sum_{y\in \bb Z} \lrc{f\lrp{\eta^{y,y+1},x}-f\lrp{\eta,x}} \\
&  + \lrc{(1-\para)\eta_x + (1-\eta_x)}\lrc{f(\eta,x+1)+f(\eta,x-1)-2f(\eta,x)}
\end{aligned}
\end{equation}
acting on local functions $ f:\bb Z \times \{0,1\}^{\bb Z} \to \bb R $ (a function $ f:\{0,1\}^{\bb Z} \to \bb R$ is called \emph{local} if $ f(\eta) $ is a function of finitely many of the variables $ \{\eta_x\}_{x\in \bb Z} $). The random walk jumps from a particle at rate $ 1-\lambda$ and from a hole at rate $ 1 $ to one of its neighbors.

For $ k\in \bb Z $ and $ \eta \in \{0,1\}^{\bb Z} $, let $ \theta_k\eta $ denote the element of $ \{0,1\}^{\bb Z} $ defined by $ (\theta_k\eta)_x = \eta_{x+k} $. We use this to define the environment process viewed from the walk $ \xi(t) = \theta_{X_t}\eta(t) $. This is a Markov process, and its generator $ L $ acts on local functions as follows:
\begin{equation}\label{key*}
\begin{aligned}
Lf(\xi) &= L^{ssep}f(\xi)  + \lrc{(1-\para)\xi_0 + (1-\xi_0)}\lrc{f(\theta_1\xi)+f(\theta_{-1}\xi)-2f(\xi)},
\end{aligned}
\end{equation}
where
\begin{equation}\label{1}
L^{ssep}f(\xi):=\sum_{y\in \bb Z} \lrc{f\lrp{\xi^{y,y+1}}-f\lrp{\xi}}
\end{equation}
is the generator of the SSEP with rate $1$.

Define the quenched probability $P^\eta(\cdot)$ on $\ZZ\times [0,\infty )$ as the probability measure of the random walk on underlying environment $\eta=\{\eta_t,t\geq 0\}$. By \eqref{generator of the joint process}, we have for $t,h\geq0$,
\begin{equation}
    P^\eta\left(X_{t+h}-X_{t}=\pm 1|X_t\right)=h\lrc{(1-\para)\eta_{X_t}(t) + (1-\eta_{X_t}(t))}+o(h).
\end{equation}
Define the annealed measure $\PP(\cdot)$ on the same space as
\begin{equation}
    \PP(\cdot)=\int P_\eta(\cdot)\dd \ann(\eta)
\end{equation}
where $\ann$ is the distribution of SSEP $\{\eta(t)\}_{t\geq0}$ with the initial distribution $\eta(0)\sim \mu$,

Our main theorem gives a quenched invariance principle of the walk with explicit scaling parameter(the variance).
\begin{theorem}\label{t1}
	Let $ (X_t,\eta(t))_{t\geq 0} $ be the Markov process generated by $ L^{joint} $,  started from $ X_0 = 0 $ and $ \eta(0) \sim \mu$. Then, for $\ann-$ almost every $\eta$, under the quenched measure $P^\eta$, the sequence of processes
	\begin{equation}\label{2}
	\lrp{\frac{X_{nt}}{\sigma(\rho)\sqrt n}:t\in [0,T]}_{n\in \bb N}
	\end{equation}
	 converges in distribution, with respect to the $J_1$ Skorohod topology, to a standard Brownian motion, where
	\begin{equation}\label{3}
	\sigma^2(\rho) = 2 - \frac{4\para\rho}{2 -\para(1-\rho)}.
	\end{equation}
\end{theorem}

This theorem will follow from the next one, which gives the asymptotic fraction of time that the walk spent on top of particles. 

\begin{theorem}\label{t2}  Keep the assumptions of Theorem \ref{t1}. Let $\xi(t) = \theta_{X_t}\eta(t)$. Then, for $\ann-$almost every $\eta$, under the quenched measure $P^\eta$,
	\begin{equation}\label{eqt2}
	\lim_{t\to\infty} \frac{1}{t}\int_0^t(2-\para \xi_0(s))(\xi_0(s) - \rho)\dd s = 0 \mbox{ in probability}.
	\end{equation}
	Or equivalently,
	\begin{equation}
	    \lim_{t\to\infty} \frac{1}{t}\int_0^t\xi_0(s)\dd s = \frac{2\rho}{2-\para+\para\rho} \mbox{ in probability}.
	\end{equation}
\end{theorem}
Theorem \ref{t2} shows the convergence under the quenched measure, which will automatically imply the same convergence result under the annealed measure. Moreover, the rate of convergence under the annealed measure has an upper bound estimation, which is also a key tool to prove Theorem \ref{t2}. This rate of convergence result is shown as follows.
\begin{theorem}\label{t3}
    Keep the assumptions of Theorem \ref{t1}. Let $\xi(t) = \theta_{X_t}\eta(t)$. For any $\epsilon>0$, there exist $T=T(\epsilon)>0$ and $C=C(\epsilon)>0$, such that for any $t>T$,
    \begin{equation}
        \PP\lrc{\frac{1}{t}\Big|\int_0^t(2-\para\xi_0(s))(\xi_0(s)-\rho)\dd s\Big|\geq \epsilon}\leq C t^{-\frac{1}{15}}.
    \end{equation}
\end{theorem}
 
\section{Proofs}
The key observation is that $ X_t $ is a mean-zero martingale with respect to the filtration generated by $ \lrp{X_t,\eta(t)}_{t\geq 0} $. Its predictable quadratic variation is given by the formula
\begin{equation}\label{4}
\lrb{X}_t = \int_0^t 2 - 2\para \xi_0(s)\dd s.
\end{equation}

More explicitly, we have 
\begin{equation}
     E^{\eta}\lrc{X_t^2 - \lrb{X}_t|(X_r, \eta(r)), r\leq s} = X_s^2 - \lrb{X}_s, \quad P^{\eta}-a.s.
\end{equation}
for any $t\geq s\geq 0$ and all $\eta$.

We claim that if $ \lim_{t\to\infty}t^{-1}\lrb{X}_t \to a  $ in probability, for some positive $ a>0 $, then the sequence $ \lrp{\frac{X_{nt}}{\sqrt n}:t\in [0,T]}_{n\in \bb N} $ converges in distribution to a Brownian motion of variance $a$, with respect to the $J_1$ Skorohod topology on the space $\mc D([0,T];\bb R)$. This follows from the Martingale Functional Central Limit Theorem, \cite{ethierkurtz} Theorem 7.1.4. 
Therefore we only need to prove that $ \lim_{t\to \infty}t^{-1}\int_0^t \xi_0(s)\dd s $ exists in probability. This follows from Theorem \ref{t2}, since if \eqref{eqt2} holds, then
\begin{equation}
    \lim_{t\rightarrow\infty}\frac{1}{t}\int_0^t\lrp{2-\lambda+\lambda\rho}\xi_0(s)\dd s=2\rho \mbox{ in probability},
\end{equation}
whence $ \lim_{t\to \infty}t^{-1}\int_0^t \xi_0(s)\dd s=\frac{2\rho}{2-\para+\para\rho}$.

Although in Theorem \ref{t2} the convergence holds quenched, we will prove the convergence in the annealed measure first. Our proof will yield a estimate on the rate of convergence that is strong enough that allows us to deduce the quenched convergence from it.

Before we start our proofs, we remind the readers that there are some technical lemmas that will be used throughout the proofs. Those lemmas are introduced in section 4 as well as their proofs. But we will use them in section 3 without mentioning too much in order to make the proof less tedious.
\subsection{Proof of the asymptotic limit of $\xi(t)$ under the annealed measure}
Our goal is to prove the following theorem.
\begin{theorem}\label{quenched t2}
    Under the assumptions of Theorem \ref{t1}, under the annealed measure $\PP$,
    \begin{equation}\label{qt2}
	\lim_{t\to\infty} \frac{1}{t}\int_0^t(2-\para \xi_0(s))(\xi_0(s) - \rho)\dd s = 0 \mbox{ in probability}.
	\end{equation}
\end{theorem}
Given $ x\in \bb Z $ and $ \ell \in \bb N $, denote
\begin{equation}\label{5}
	\vec{\xi}_x^{\ell}:=\frac{\xi_{x+1}+\cdots + \xi_{x+\ell}}{\ell},\quad  \lvec{\xi}_x^{\ell}:=\frac{\xi_{x-1}+\cdots + \xi_{x-\ell}}{\ell}.
\end{equation}

For any choice of positive integers $ \ell $ and $ n $ one can write
\begin{align}\label{decomposition}
	& \frac{1}{t}\int_0^t(2-\para \xi_0(s))(2\xi_0(s) - 2\rho)\dd s \\
	= & \frac{1}{t}\int_0^t(2-\para \xi_0(s))(2\xi_0(s) - \xi_n(s) - \xi_{-n}(s))\dd s \\
	+ & \frac{1}{t}\int_0^t(2-\para \xi_0(s))(\xi_n(s) - \vec\xi^{\ell}_n(s) + \xi_{-n}(s)-\lvec\xi^{\ell}_{-n}(s) )\dd s\\
	+ & \frac{1}{t}\int_0^t(2-\para \xi_0(s))(\vec\xi_n^{\ell}(s) + \lvec\xi_n^{\ell}(s) - 2\rho)\dd s \label{eq:lateral_dec_1}
	\end{align}
	We are going to choose $ n $ and $ \ell $ depending on $ t $ in such a way that all three integrals on the right-hand side converge to $ 0 $ in probability, as $ t\to \infty $. It turns out one can choose
	\begin{equation}\label{ass:ell_and_n}
	n = \lfloor t^{\alpha}\rfloor \mbox{ for some }\alpha \in (\frac{1}{2},\frac{2}{3}), \quad 1\ll \ell \ll \frac{t}{n}.
	\end{equation}

\begin{proposition}\label{prop:replacement_1}
	Under the assumption of Theorem \ref{t1}, assume \eqref{ass:ell_and_n}. Under the annealed measure
\begin{equation}\label{33}
\lim_{t\rightarrow\infty}\frac{1}{t}\int_0^t(2-\para\xi_0(s))(2\xi_0(s)-\xi_n(s)-\xi_{-n}(s))ds=0
\end{equation}
in probability.
\end{proposition}

 The proof strategy is to show that the integrand is in the range of the generator and use this to rewrite the integral as the sum of a martingale and a vanishing term. The martingale is then shown to vanish too, by means of an explicit bound on its quadratic variation.

Thus we seek a function $ \psi_{n,\ell} $ such that $ L\psi_{n,\ell}(\xi)=(2-\para\xi_0)(2\xi_0-\xi_n-\xi_{-n}) $. We start the search by computing
\begin{equation}\label{eq:letax}
\begin{aligned}
L\xi_x & = \lrc{\xi_{x+1}+\xi_{x-1}-2\xi_x}\lrc{\xi_0(1-\para)+(1-\xi_0)}+\lrp{\xi_{x+1}-\xi_x}+\lrp{\xi_{x-1}-\xi_x} \\
& = (2-\para \xi_0)(\xi_{x+1}+\xi_{x-1}-2\xi_x).
\end{aligned}
\end{equation}
Let $ k > 0$. Sum from $ x=-k+1 $ to $ x = k-1 $ to get
\begin{equation}\label{eq:gradient_intherange}
L\lrp{\sum_{x=-k+1}^{k-1}\xi_x} = (2-\para \xi_0)\lrp{\xi_k - \xi_{k-1} + \xi_{-k}-\xi_{-k+1} }.
\end{equation}
Sum from $ k= 1 $ to $ k=n $ to get
\begin{equation}\label{6}
L\lrp{\sum_{k=1}^n\sum_{x=-k+1}^{k-1}\xi_x  } = (2-\para \xi_0)\lrp{-2\xi_0 + \xi_n + \xi_{-n} }.
\end{equation}
Define
\begin{equation}\label{def:psi}
\psi_{n,\ell}(\xi) := -\sum_{k=1}^n\sum_{x=-k+1}^{k-1}(\xi_x - \rho),
\end{equation}
the following process is a mean zero martingale with respect to the filtration generated by $ \xi(s)_{s\geq 0} $:
\begin{equation}\label{def:mart_psi}
M_s(\psi_{n,\ell}):= \psi_{n,\ell}(\xi(s)) - \psi_{n,\ell}(\xi(0)) - \int_0^s (2-\para\xi_0(r))(2\xi_0(r)-\xi_n(r)-\xi_{-n}(r))\dd r.
\end{equation}

We need separate arguments to control the terms $ \frac{\psi_{n,\ell}(\xi(t)) - \psi_{n,\ell}(\xi(0))}{t} $ and $ \frac{1}{t}M_t(\psi_{n,\ell}) $.
\begin{lemma} \label{l3.3}
Under the assumptions of Theorem \ref{t1}, assume \eqref{ass:ell_and_n}. With $ \psi_{n,\ell} $ given by \eqref{def:psi},
	\[ \lim_{t\to \infty} \frac{1}{t}\bb E |\psi_{n,\ell}(\xi(t))| = 0 . \]
\end{lemma}

\begin{proof}
	Rewrite $ \psi_{n,\ell}(\xi) = n(\xi_0 - \rho) + \sum_{k=1}^n (n-k)(\xi_k + \xi_{-k}-2\rho) $. It suffices to prove
	\begin{equation}\label{eq:bounding_psi}
	\lim_{t\to \infty}\frac{1}{t}\bb E\Big|\sum_{k=1}^n (n-k)(\xi_k(t) - \rho) \Big| = 0.
	\end{equation}
	Notice that the trivial pointwise bound is of order $ n^2 $, which is much bigger than $ t $. The idea is that when $ k $ is large the variables $ \xi_k(t)-\rho $ are approximately independent and have mean zero. Recall that $ \xi_x(t) = \eta_{x+X_t}(t) $, where $ \eta(t) $ is a stationary SSEP and $ X_t $ is the random walk. Then
	
	\begin{align}
	\bb E \Big| \sum_{k=1}^n (n-k)(\xi_k(t) - \rho) \Big|
	& \leq n^2\bb P\lrp{|X_t| > n}
	+ \bb E\Big|\sup_{|j|\leq n} \sum_{k=1}^n (n-k)\lrp{\eta_{k+j}(t)-\rho} \Big|.
	\end{align}
	By Lemma \ref{lem:rw_max}, the first term is of order $ t^3n^{-4} $. It then follows from our assumption \eqref{ass:ell_and_n} that $ \lim_{t\to\infty}t^{-1}n^2\bb P\lrp{|X_t| > n} =0 $, as we need.
	
	To bound the second term, write
	
	\begin{equation}
	\begin{aligned}
	& \frac{1}{t}\bb E\Big|\sup_{|j|\leq n} \sum_{k=1}^n (n-k)\lrp{\eta_{k+j}(t)-\rho} \Big| \\
	 = &\int_0^{\infty} \bb P\lrp{\Big|\sup_{|j|\leq n} \sum_{k=1}^n (n-k)\lrp{\eta_{k+j}(t)-\rho} \Big| > \beta t} \dd \beta \\
	 \leq &\delta + \sum_{|j|\leq n}\int_\delta^{\infty} \bb P\lrp{\Big|\sum_{k=1}^n (n-k)\lrp{\eta_{k+j}(t)-\rho} \Big| > \beta t} \dd \beta \\
	 \leq &
	\delta + 2\sum_{|j|\leq n}\int_\delta^{\infty} \bb \exp\lrp{-\frac{t^2}{n^3}\frac{\beta^2}{2}} \dd \beta \\
	 \leq & \delta + 12\frac{n^{\frac{5}{2}}}{t}\cdot\exp\lrp{-\frac{t^2}{n^3}\frac{\delta^2}{2}} \\
	 =& \delta + 12\, t^{\frac{5\alpha}{2}-1}\cdot\exp\lrp{-\frac{\delta^2 t^{2-3\alpha}}{2}}.
	\end{aligned}
	\end{equation}
	The fourth line is by Lemma \ref{lem:subgauss}, the fifth line is by lemma \ref{eq:gauss_tail}, and the last line is by \eqref{ass:ell_and_n}.
	
	Now choose $ \delta = t^{-(\frac{2}{3}-\alpha)} $, we get an upper bound of $\bb E \Big| \sum_{k=1}^n (n-k)(\xi_k(t) - \rho) \Big|$ as
	\begin{equation}\label{aa}
	    \frac{1}{t}\bb E \Big| \sum_{k=1}^n (n-k)(\xi_k(t) - \rho) \Big|\leq c_0\left( t^{2-4\alpha}+t^{\alpha-\frac{2}{3}}\right)
	\end{equation}
	for some constant $c_0>0$ and $t$ large enough. Let $ t\to \infty $, the right hand side converges to zero, this finishes the proof of \eqref{eq:bounding_psi}.
\end{proof}

The next lemma controls $\frac{1}{t}M_t(\psi_{n,\ell})$.

\begin{lemma} \label{l3.4}
Under the assumptions of Theorem \ref{t1}, assume \eqref{ass:ell_and_n}. With $ \psi_{n,\ell} $ given by \eqref{def:psi} and $ M_t(\psi_{n,\ell}) $ given by \eqref{def:mart_psi},
	\begin{equation}\label{7}
	\lim_{t\to \infty} t^{-2} \bb E\lrc{M_t^2(\psi_{n,\ell})} = 0.
	\end{equation}
\end{lemma}	

\begin{proof}
	There is an explicit formula for the predictable quadratic variation of $ M_t(\psi_{n,\ell}) $:
	\begin{align}
	\lrb{M_{\cdot}(x)}_t
	& = \int_0^t \sum_{x\in \bb Z} \lrc{\psi_{n,\ell}\lrp{\xi^{x,x+1}(s)} - \psi_{n,\ell}\lrp{\xi(s)} }^2 \dd s \\ \label{eq:quadvar_psi_ex}
	& + \int_0^t (1-\para \xi_0(s))\lrc{\psi_{n,\ell}\lrp{\theta_1\xi(s)} - \psi_{n,\ell}\lrp{\xi(s)} }^2 \dd s \\ \label{eq:quadvar_psi_rw}
	& + \int_0^t (1-\para \xi_0(s))\lrc{\psi_{n,\ell}\lrp{\theta_{-1}\xi(s)} - \psi_{n,\ell}\lrp{\xi(s)} }^2 \dd s.
	\end{align}
	Our goal is to prove $ \lim_{t\to \infty}t^{-2}\bb E\lrb{M_{\cdot}(\psi_{n,\ell})}_t = 0 $.  To bound the first term, notice that $ \lrc{\psi_{n,\ell}\lrp{\xi^{x,x+1}} - \psi_{n,\ell}\lrp{\xi} }^2 = 0 $ if $ |x|>n $ and no greater than $  1 $ if $ |x|\leq n $, so the integrand is much smaller than $ 2tn $. The second term demands more work while the third term has the similar proof as the second one. To start, we compute
	\begin{equation}\label{8}
	-\psi_{n,\ell}(\theta_1\xi)+\psi_{n,\ell}(\xi) = \sum_{k=1}^{n}\xi_k-\sum_{k=-n+1}^0\xi_k.
	\end{equation}
	It is enough to prove
	\begin{equation}\label{eq:bound_quadvar_rw}
	\lim_{t\to \infty}\sup_{s\leq t}t^{-1}\bb E\lrc{\lrp{\sum_{k=1}^{n}\xi_k(s)-\sum_{k=-n+1}^0\xi_k(s)}^2} = 0.
	\end{equation}
	The expectation above is small by the same reason that \eqref{eq:bounding_psi} is small: the random variables $ \xi_k(s) $, for large $ k $, are approximately independent of mean $ \rho $. We follow the same method of proof.
	\begin{equation}\label{9}
	\begin{aligned}
	& t^{-1}\bb E\lrc{\lrp{\sum_{k=1}^{n}\xi_k(s)-\sum_{k=-n+1}^0\xi_k(s)}^2} \\
	& \leq \frac{n^2}{t}\bb P\lrp{|X_s|> n } + t^{-1}\bb E\lrc{\sup_{|j|\leq n}\lrp{\sum_{k=1}^{n}\eta_{k+j}(s)-\sum_{k=-n+1}^0\eta_{k+j}(s)}^2 }
	\end{aligned}
	\end{equation}
	By Lemma \ref{lem:rw_max}, the first term is of order $\frac{t^2}{n^4} $, so it vanishes as $ t\to \infty $. The second term is bounded, for any $ \delta >0 $, by
	\begin{equation}\label{10}
	\begin{aligned}
	& \delta + \sum_{|j|\leq n}\int_\delta^\infty \bb P\lrp{\lrc{\sum_{k=1}^{n}\eta_{k+j}(s)-\sum_{k=-n+1}^0\eta_{k+j}(s)}^2 \geq \beta t}\dd \beta \\
	\leq & \,\delta + 2\sum_{|j|\leq n}\int_\delta^\infty \exp\lrp{-\frac{\beta t}{10n }}\dd \beta  \\
	\leq & \, \delta + \frac{60 n^2}{t}\exp\lrp{-\frac{\delta t}{10n}} \\
	 = &\,\delta + 60 t^{2\alpha - 1}\exp\lrp{-\frac{\delta t^{1-\alpha}}{10}} .
	\end{aligned}
	\end{equation}
The second line is by Lemma \ref{lem:subgauss}. Choose $ \delta =t^{-\frac{1-\alpha}{2}} $, we then get an upper bound 
\begin{equation}
    t^{-1}\bb E\lrc{\lrp{\sum_{k=1}^{n}\xi_k(s)-\sum_{k=-n+1}^0\xi_k(s)}^2}\leq c_1\left(\frac{t^2}{n^4}+t^{\frac{\alpha-1}{2}}\right)
\end{equation}
for some constant $c_1>0$ and $t$ large enough.

Collect all the above upper bounds we have 
\begin{equation}\label{bb}
    \frac{1}{t^2}\bb E\lrc{M_t^2(\psi_{n,\ell})}\leq \frac{2n}{t}+c_1\left(\frac{t^2}{n^4}+t^{\frac{\alpha-1}{2}}\right)=2t^{\alpha-1}+c_1(t^{2-4\alpha}+t^{\frac{\alpha-1}{2}}).
\end{equation}
By the assumption \ref{ass:ell_and_n}, the upper bound vanishes as $t\rightarrow\infty$.
%
	
\end{proof}

\begin{proof}[Proof of Proposition \ref{prop:replacement_1}] By Chebyshev inequality, for any $\epsilon>0$, notice that $\xi(0)=\eta(0)$, there exists some constant $c_2>0$ such that
\begin{equation}\label{c}
    \begin{aligned}
    \PP\lrc{\Big|\frac{\psi_{n,\ell}(\xi(0))}{t}\Big|\geq \epsilon}\leq\frac{\EE\lrc{\psi_{n,\ell}^2(\xi(0))}}{\epsilon^2 t^2}
    =\frac{\EE\lrc{\lrp{\sum_{k=1}^n (n-k)(\eta_k(0) - \rho)}^2}}{\epsilon^2t^2}
    \leq \frac{c_2}{\epsilon^2}t^{2\alpha-2}.
    \end{aligned}
\end{equation}
The last inequality uses the fact that $\{\eta_k(0)-\rho\}_{k\in \ZZ}$ is an i.i.d mean zero sequence. The cross terms above will vanish after taking the expectation.

Use this upper bound, together with \eqref{def:mart_psi}, \eqref{aa}, and \eqref{bb} for any $\epsilon>0$,
\begin{equation}\label{key_1}
    \PP\lrc{\frac{1}{t}\Big|\int_0^t(2-\para\xi_0(s))(2\xi_0(s)-\xi_n(s)-\xi_{-n}(s))\dd s\Big|\geq \epsilon}\leq C_0(\epsilon) t^{\gamma_1}
\end{equation}
where constant $C_0(\epsilon)>0$ and 
\begin{equation}
    \gamma_1=\max\left\{2\alpha-2,\alpha-1,\frac{\alpha-1}{2},2-4\alpha,\alpha-\frac{2}{3}\right\}<0
\end{equation}
due to assumption \eqref{ass:ell_and_n}. Hence Proposition \ref{prop:replacement_1} is proved.
\end{proof}

The next proposition shows the limit of the second part of the decomposition \eqref{decomposition}. 
\begin{proposition}\label{replancement_2}
	Under the assumption of Theorem \ref{t1}, assume \eqref{ass:ell_and_n}. Under the annealed measure,
	\begin{equation}\label{lem:replacement_2}
	\lim_{t\to \infty}\frac{1}{t}\int_0^t(2-\para \xi_0(s))(\xi_n(s) - \vec\xi^{\ell}_n(s) + \xi_{-n}(s)-\lvec\xi^{\ell}_{-n}(s) )\dd s = 0
	\end{equation}
	in probability.
\end{proposition}

\begin{proof}We show that the integrand is in the range of the generator and split the integral into a martingale term plus a vanishing term.
Notice that
\begin{equation}\label{12}
\xi_x - \vec\xi_{x}^{\ell} = \sum_{j=0}^{\ell - 1}\frac{\ell - j}{\ell}(\xi_{x+j}-\xi_{x+j+1})
\end{equation}
and
\begin{equation}\label{13}
\xi_x - \lvec\xi_{x}^{\ell} = \sum_{j=0}^{\ell - 1}\frac{\ell - j}{\ell}(\xi_{x-j}-\xi_{x-j-1}).
\end{equation}

From \eqref{eq:gradient_intherange}, we get

\begin{equation}\label{14}
(2-\para\xi_0)\lrp{\xi_n - \vec\xi_n^{\ell}+\xi_{-n}-\lvec\xi_{-n}^{\ell} } =
-L\varphi_{n,\ell}(\xi),
\end{equation}
where
\begin{equation}\label{15}
\varphi_{n,\ell}(\xi) := \sum_{j=0}^{\ell -1}\frac{\ell - j}{\ell}\sum_{x=-n-j}^{n+j}\xi_x.
\end{equation}

The process

\begin{align}\label{mart-rep-2}
M_s(\varphi_{n,\ell})  := \varphi_{n,\ell}(\xi(s)) - \varphi_{n,\ell}(\xi(0))
 - \int_0^s L\varphi_{n,\ell}(\xi(r)) \dd r
\end{align}
is a martingale with respect to the filtration generated by $ \lrp{\xi(s)}_{s\geq 0} $. To prove \eqref{lem:replacement_2}, we show that $ |\varphi_{n,\ell}| \ll t $ and $ \lrb{M_{\cdot}(\varphi_{n,\ell})}_t \ll t^2 $. For the first term,
\begin{equation}\label{a}
|\varphi_{n,\ell}(\xi)| \leq \sum_{j=0}^{\ell -1}\frac{\ell - j}{\ell}\lrp{2n + 2j +1} \leq C\lrp{\ell n + \ell^2}
\end{equation}
for some $ C>0 $, so it follows from \eqref{ass:ell_and_n} that $ \lim_{t\to \infty}t^{-1}|\varphi_{n,\ell}(\xi)| = 0 $ for any $ \xi \in \{0,1\}^{\bb Z} $.

It remains to prove that $ t^{-1}M_t(\varphi_{n,\ell}) \to 0$ in probability. We prove this by controlling the second moment of $ M_t(\varphi_{n,\ell}) $ through its predictable quadratic variation
\begin{equation}\label{16}
\begin{aligned}
\lrb{M_{\cdot}(\varphi_{n,\ell})}_t & = \int_0^t \sum_{x\in \bb Z}\lrc{\varphi_{n,\ell}\lrp{\xi^{x,x+1}(s)} - \varphi_{n,\ell}\lrp{\xi}(s)}^2\dd s + \\
& + \int_0^t \lrp{1-\para\xi_0(s)}\lrc{\varphi_{n,\ell}\lrp{\theta_1\xi(s)} - \varphi_{n,\ell}\lrp{\xi(s)}}^2 \dd s \\
& + \int_0^t \lrp{1-\para\xi_0(s)}\lrc{\varphi_{n,\ell}\lrp{\theta_{-1}\xi(s)} - \varphi_{n,\ell}\lrp{\xi(s)}}^2 \dd s.
\end{aligned}
\end{equation}

We claim that, 

\begin{equation}\label{eq:replacement_2_quadvar}
\lim_{t\to \infty}t^{-2} \lrb{M_{\cdot}(\varphi_{n,\ell})}_t = 0 .
\end{equation}
Let $ a_k := \sum_{j=k}^{\ell-1}\frac{\ell - j}{\ell} $. Then

\begin{equation}\label{17}
\begin{aligned}
\varphi_{n,\ell}(\xi) & = a_0\sum_{j=-n}^{n}\xi_j + \sum_{k=1}^{\ell - 1}a_k\lrp{\xi_{n+k}+\xi_{-n-k}}.
\end{aligned}
\end{equation}
It's easy to see that
\begin{equation}\label{18}
\lrc{\varphi_{n,\ell}\lrp{\xi^{x,x+1}} - \varphi_{n,\ell}\lrp{\xi}}^2 \leq \sum_{k=0}^{\ell -1} \Ind{|x|=n+k}(a_k - a_{k+1})^2 \leq \ell
\end{equation}
and
\begin{equation}\label{19}
\lrc{\varphi_{n,\ell}\lrp{\theta_1\xi} - \varphi_{n,\ell}\lrp{\xi}}^2 \leq (2a_0)^2 \leq C\ell^2
\end{equation}
for some $ C>0 $ independent of $ \ell $ and $ n $.

These bounds imply 
\begin{equation}\label{b}
     \lrb{M_{\cdot}(\varphi_{n,\ell})}_t \leq  3\lrp{tl + Ct\ell^2}.
\end{equation}
Hence, by \eqref{mart-rep-2}, \eqref{a}, \eqref{b}, for any $\epsilon>0$,
\begin{equation}\label{key_2}
    \PP\lrc{\frac{1}{t}\Big|\int_0^t(2-\para \xi_0(s))(\xi_n(s) - \vec\xi^{\ell}_n(s) + \xi_{-n}(s)-\lvec\xi^{\ell}_{-n}(s) )\dd s\Big|\geq \epsilon}\leq C_1(\epsilon) \frac{l+ln+l^2}{t}
\end{equation}
for some $C_1(\epsilon)>0$. By \eqref{ass:ell_and_n} the right hand side of \eqref{key_2} indeed converges to zero.
\end{proof}
\begin{remark}
Proposition \ref{replancement_2} gives the convergence under the annealed measure. But one can see from the key upper bounds \eqref{a} and \eqref{b} are deterministic. This implies that the convergence holds not only in the annealed sense, but also in the quenched sense, i.e. under $P^\eta$ for all $\eta\in [0,1]^\NN\times \RR^+$.
\end{remark}
\begin{proposition}\label{ld}
Under the assupmtion of Theorem \ref{t1}, assume \eqref{ass:ell_and_n}. Under the annealed measure
	\begin{equation}\label{20}
	\frac{1}{t}\int_0^t (2-\para \xi_0(s))(\vec\xi^{\ell}_n(s)-\rho)\dd s \to 0
	\end{equation}
	in probability as $ t\to \infty $. The same holds if $ \vec\xi^{\ell}_n $ is replaced by $ \lvec\xi^{\ell}_{-n} $.
\end{proposition}

\begin{proof}
	Define, for $ m > 0 $, the event
	\begin{equation}\label{21}
	A_m :=\lrch{ \max_{s\leq t} |X_s| < m}.
	\end{equation}
	
	Then
	\begin{equation}\label{eq:latdec01}
	\begin{aligned}
	& \bb E\lrc{\lrp{\frac{1}{t} \int_0^t(2-\para \xi_0(s))(\vec\xi^{\ell}_n(s)-\rho)\dd s }^2}  \\
	& \leq 4\bb P(A_m^c) + \frac{1}{t}\int_0^t \bb E\lrc{\ind{A_{m}}(2-\para \xi_0(s))^2(\vec\xi^{\ell}_n(s)-\rho)^2}\dd s
	\end{aligned}
	\end{equation}
	
	We will prove that, for $ t^{\frac{1}{2}} \ll m \ll n $, the upper bound in the last equation vanishes as $ t\to \infty $.
	To bound the second term, we apply the  Lateral Decoupling Lemma (\cite{hkt}, Proposition 4.1). To do so, we need the random variable inside the expectation to be a function of the exclusion process only. Thus we rewrite the expectation as
	\begin{equation}\label{eq:lat_dec_lem}
	\begin{aligned}
	& \bb E\lrc{\ind{A_{m}}(2-\para \xi_0(s))^2(\vec\xi^{\ell}_n(s)-\rho)^2} \\
	 & = \sum_{|k| <m}\bb E\lrc{ (2-\para \eta_k(s))^2 \bb E\lrp{\Ind{A_{m}, X_s = k}|\mc F_t}(\vec\eta^{\ell}_{n+k}(s)-\rho)^2},
	\end{aligned}
	\end{equation}
	where $ \mc F_t $ is the filtration generated by $ (\eta_s)_{s \in [0,t]} $. If $ m \ll n $ and \eqref{ass:ell_and_n} holds, we can apply Proposition \ref{prop:lat_decoupling} with $ H = t $,  $ f_1(\eta) = (2-\para \eta_k(s))^2\frac{1}{2}\bb E\lrp{\Ind{A_{m}, X_s = k}|\mc F_t} $ and $ f_2(\eta) = (\vec\eta^{\ell}_{n+k}(s)-\rho)^2 $ for all $|k|<m$. Note that the support of $f_1$ is contained in $[-m,m]\times [0,t] \subset [m-t,m]\times [0,t]$, 
	and the support of $f_2$ is contained in $[n+k,n+k+\ell]\times[0,t] \subset [n+k,n+k+t]\times[0,t]$,
	and by \eqref{ass:ell_and_n} and the assumption that $t^{1/2} \ll m \ll n$ the horizontal separation of these boxes is $n+k-m \gg t^{\alpha'}$ for any $\alpha' \in(\frac{1}{2},\alpha)$. 
	Therefore, applying Proposition \eqref{prop:lat_decoupling} it holds that
	\begin{equation}\label{22}
	\begin{aligned}
	& \bb E\lrc{\ind{A_{m}}(2-\para \xi_0(s))^2(\vec\xi^{\ell}_n(s)-\rho)^2} \\
	 =& \sum_{|k| <m}\bb E\lrc{ (2-\para \eta_k(s))^2 \bb E\lrp{\Ind{A_{m}, X_s = k}|\mc F_t}(\vec\eta^{\ell}_{n+k}(s)-\rho)^2}\\
	 \leq& \, 4m\cdot\exp\lrp{-t^{2\alpha' -1}} + \sum_{|k| <m}  \bb E\lrc{\Ind{A_{m}, X_s = k}(2-\para \eta_k(s))^2} \bb E\lrc{(\vec\eta^{\ell}_{n+k}(s)-\rho)^2} \\
	 \leq& \, 4m\cdot\exp\lrp{-t^{2\alpha' -1}} + \frac{4}{\ell}.
	\end{aligned}
	\end{equation}
	
	Using this bound in \eqref{eq:lat_dec_lem}, together with Lemma \eqref{lem:rw_max}, we get
	
	\begin{equation}\label{key_3}
	\begin{aligned}
		 \bb E\lrc{\lrp{\frac{1}{t} \int_0^t(2-\para \xi_0(s))(\vec\xi^{\ell}_n(s)-\rho)\dd s }^2}  
	  \leq  4m\cdot\exp\lrp{-t^{2\alpha' - 1}} + \frac{4}{\ell} +  \frac{c_3t^3}{m^6}
	\end{aligned}
	\end{equation}
	for some $c_3>0$ and $t$ large enough.
	If $ t^{\frac{1}{2}} \ll m \ll n $ and \eqref{ass:ell_and_n} holds then the upper bound vanishes as $ t \to \infty $. By Chebyshev inequality, this finishes the proof.

\end{proof}
One can get Theorem \ref{quenched t2} immediately from propositions \ref{prop:replacement_1}, \ref{replancement_2}, and \ref{ld}.

\subsection{Proof of the asymptotic limit of $\xi(t)$ under the quenched measure}

First recall \eqref{key_1}, \eqref{key_2} and \eqref{key_3}, all these inequalities imply that the convergence in probability holds not only the converge in probability does hold under the annealed measure, but one can also have a polynomially decreasing upper bound by choosing adequate $\alpha,\ell, m$. 
\begin{proof}[Proof of Theorem \ref{t3}]
    Let $\alpha=0.6$, $\ell=t^{0.2}$ and $m=t^{0.55}$, for any $\epsilon>0$ and $t$ large enough, one can check
\begin{equation}\label{annealed-ub}
    \PP\lrc{\frac{1}{t}\Big|\int_0^t(2-\para\xi_0(s))(\xi_0(s)-\rho)\dd s\Big|\geq \epsilon}\leq C(\epsilon) t^{-\frac{1}{15}}
\end{equation}
for some $C(\epsilon)>0$.
\end{proof}

The next lemma shows how to get the convergence in probability under the quenched measure $\ann-a.s.$ from the annealed measure.
\begin{lemma}\label{quenched lemma}
    Under the assumptions of Theorem \ref{t1}, let $Y_t=\int_0^t(2-\para\xi_0(s))(\xi_0(s)-\rho)\dd s$ for $t>0$ and $\ann$ defined in section 2. Then for any $\epsilon,\delta>0$, there exists $t_{\eta}(\epsilon,\delta)>0$ such that
    \begin{equation}
        \ann\lrc{\lrch{P^\eta\lrp{|Y_t|\geq \epsilon t}< \delta}\mbox{ for}~\forall t>t_{\eta}(\epsilon,\delta)}=1.
    \end{equation}
\end{lemma}
\begin{proof}
    Define a sequence $\{t_k\}_{k\geq1}$ as $t_k=k^{16}$. By \eqref{annealed-ub}, we have for $k$ large enough,
    \begin{equation}
        \PP\lrc{|Y_{t_k}|\geq \epsilon t_k}\leq C(\epsilon)k^{-\frac{16}{15}}.
    \end{equation}
    By Chebyshev inequality,
    \begin{equation}
        \ann\lrc{P^\eta\lrc{|Y_{t_k}|\geq\epsilon t_k}\geq \delta}\leq \frac{1}{\delta}\PP\lrc{|Y_{t_k}|\geq \epsilon t_k}\leq \frac{C(\epsilon)}{\delta}k^{-\frac{16}{15}}.
    \end{equation}
    the upper bound is summable for $k$. Thus by Borel-Cantelli lemma,
    \begin{equation}
        \ann\lrc{\lrch{P^\eta\lrc{|Y_{t_k}|\geq\epsilon t_k}\geq \delta}i.o.}=0.
    \end{equation}
    For any $t\geq 1$, it must lie in the interval $[t_k,t_{k+1})$ for some $k$. Notice that $Y_t$ has bounded increments, which means $|Y_s-Y_r|\leq 2|s-r|$ for any $s,r>0$. This gives the upper bound
    \begin{equation}
        \frac{|Y_t|}{t} \leq\frac{|Y_{t_k}|+2(t_{k+1}-t_k)}{t_k}.
    \end{equation}
    Let $k_\epsilon>0$ satisfy $2(t_{k_\epsilon+1}-t_{k_\epsilon})t_{k_\epsilon}^{-1}<\epsilon$, for any $k>k_\epsilon$ and $t\in[t_k,t_{k+1})$, $\{|Y_t|/t\geq 2\epsilon\}$ implies $\{|Y_{t_k}|/t_k\geq\epsilon\}$. Define $A_{\epsilon,\delta}$ that $A_{\epsilon,\delta}^c=\lrch{\lrch{P^\eta\lrc{|Y_{t_k}|\geq\epsilon t_k}\geq \delta}i.o.}$. Choose any $\eta\in A_{\epsilon,\delta}$, there exists $k_\eta(\epsilon,\delta)$ such that for all $k>k_\eta(\epsilon,\delta)$
    \begin{equation}
        P^\eta\lrc{|Y_{t_k}|\geq\epsilon t_k}<\delta.
    \end{equation}
    Pick $t_\eta(2\epsilon,\delta)=t_{k_\epsilon}\vee t_{k_\eta(\epsilon,\delta)}$ then by the above argument we have for all $t\in[t_k,t_{k+1}),k\geq k_\epsilon\vee k_{\eta}(\epsilon,\delta)$,
    \begin{equation}
        P^\eta\lrc{|Y_{t}|\geq2\epsilon t}\leq P^\eta\lrc{|Y_{t_k}|\geq\epsilon t_k}<\delta
    \end{equation}
    which finishes the proof since $P_\mu\lrp{A_{\epsilon,\delta}}=1$.
\end{proof}
In the last part of this section we prove Theorem \ref{t2}.
\begin{proof}[Proof of Theorem \ref{t2}] 
    From Lemma \ref{quenched lemma}, we just need one more step to reach our final goal. To see this, for any $\epsilon>0$, let 
    \begin{equation}
        A_\epsilon=\bigcap_{n=1}^\infty A_{\epsilon,\frac{1}{n}}.
    \end{equation}
    We have $P_\mu\lrp{A_\epsilon}=1$ since it is a intersection of countably many sets while each has probability 1. Choose any $\eta\in A_\epsilon$, for any $n\geq1$, 
    \begin{equation}
        P^\eta\lrc{|Y_{t}|\geq\epsilon t}<\frac{1}{n}
    \end{equation}
    holds for all $t>t_\eta(\epsilon,\frac{1}{n})$. Thus  $t^{-1}|Y_t|$ converge to zero in probability under $P^\eta$.
\end{proof}

\section{Technical lemmas}

\begin{lemma}[\cite{lugosi}, Theorem 2.8]\label{lem:subgauss}
	Let $ \zeta_1, \zeta_2, \ldots $ be i.i.d. random variables with $ |\zeta_1|\leq 1 $ and $ \bb E\zeta_1 =0 $. Then, for any $ \lambda > 0 $,
	\begin{equation}\label{23}
	\bb P\lrp{\Big| \sum_{j\leq n} b_j \zeta_j \Big| > \lambda} \leq 2\cdot\exp\lrp{-\frac{\lambda^2}{2\sum_{j\leq n}b_j^2}}.
	\end{equation}
\end{lemma}

\begin{lemma}\label{lem:gauss_tail}
	For any $ \delta > 0 $,
	\begin{equation}\label{eq:gauss_tail}
	\int_{\delta}\exp\lrp{-\frac{x^2}{2\sigma^2}}\dd x \leq \sqrt {2\pi \sigma^2}\cdot \exp\lrp{-\frac{\delta^2}{2\sigma^2}}.
	\end{equation}
\end{lemma}

\begin{proof}
	For any $ \lambda >0 $,
	\begin{equation}\label{24}
	\begin{aligned}
	\int_{\delta}\exp\lrp{-\frac{x^2}{2\sigma^2}}\dd x
	& \leq  e^{-\lambda \delta} \int_{\delta}^\infty \exp\lrp{\lambda x-\frac{x^2}{2\sigma^2} } \dd x \\
	& \leq \exp\lrp{-\lambda \delta + \frac{\lambda^2\sigma^2}{2} } \int_{-\infty}^\infty \exp\lrp{-\frac{\lrp{x - \lambda \sigma^2}^2}{2\sigma^2}} \dd x \\
	& = \sqrt {2\pi \sigma^2}\exp\lrp{-\lambda \delta + \frac{\lambda^2\sigma^2}{2}}.
	\end{aligned}
	\end{equation}
	Choosing $ \lambda = \delta/\sigma^2 $ gives the desired bound.
\end{proof}

\begin{lemma}\label{lem:rw_max}For any positive $ \gamma $ and $ t $,
	\begin{equation}
	\bb P \lrp{\sup_{s\leq t} |X_s| \geq \gamma } = O\lrp{\frac{t^3}{\gamma ^6}}.
	\end{equation}
\end{lemma}

\begin{proof}
	The first observation is that $ X $ is a martingale, so Doob's $ L^p -$inequality gives
	\begin{equation}\label{eq:bound_bad_event}
	\bb P\lrp{\sup_{s\leq t} |X_s| \geq \gamma} \leq \lrp{\frac{6}{5}}^6\frac{\bb E\lrp{X_t^6}}{\gamma^6}.
	\end{equation}

	To bound the sixth moment, we compare our random walk with a simple symmetric walk: let $ Y_1, \ldots, Y_n $ be i.i.d. random variables with $ \bb P\lrp{Y_1 = \pm 1} = 1/2 $ and let $ J_t $ denote the number of times that $ X $ jumps during the time interval $ [0,t] $. Then $ X_t = \sum_{k=1}^{J_t}Y_k $ in distribution, whence
	\begin{equation}\label{eq:4th_moment}
	\begin{aligned}
	\bb E \lrp{X_t^6}
	& = \bb E \lrp{\sum_{i\leq J_t}Y_i^6 + 15\sum_{i<j\leq J_t}Y_i^2Y_j^4 + 90\sum_{i<j<k\leq J_t}Y_i^2Y_j^2Y_k^2} \\
	& \leq \bb E \lrp{J_t + 15J_t^2 + 90 J_t^3}.
	\end{aligned}
	\end{equation}
	Since $ J_t $ is stochastically dominated by a mean $ t $ Poisson random variable, the last expectation is bounded by a multiple of $ t^3 $.
\end{proof}

The next lemma comes from \cite{hkt}. To get the version stated below, one only needs to change the last line of the original proof, using \eqref{eq:gauss_tail}.

\begin{proposition}[Lateral Decoupling, \cite{hkt} Proposition 4.1]\label{prop:lat_decoupling}
	Let $ f_1, f_2:\{0,1\}^{\bb Z}\times \bb R_+ \to [0,1] $ be measurable functions and $ H, y, \alpha >0 $. Let $ B_1 = [-H,0]\times [0,H] \subset \bb R^2 $ and $ B_2 = [y, y+H]\times [0,H] \subset \bb R^2 $. Assume $ f_1 $ is supported on $ B_1 $, that is, if the trajectories $ \eta, \eta':\bb Z \times \bb R_+ \to \{0,1\} $ satisfy $ \eta_x(s) = \eta'_x(s) $ for all $ (x,s)\in B_1 $ then $ f_1(\eta)=f_1(\eta') $. Assume $ f_2 $ is supported on $ B_2 $. Finally, denote by $ \bb P_{\rho} $ the law of SSEP started from equilibrium at density $ \rho \in (0,1) $, that is, started from the product measure $ \otimes_{x\in \bb Z}\mathrm{Ber}(\rho) $. Let $ \bb E_{\rho} $ be the expectation with respect to $ \bb P_{\rho} $.
	
	Then $ y \geq H^{\alpha} $ implies
	\begin{equation}\label{25}
	\bb E_{\rho}\lrc{f_1f_2}\leq \bb E_{\rho}\lrc{f_1}\bb E_{\rho}\lrc{f_2} + C\,e^{-H^{2\alpha - 1}}
	\end{equation}
	for some $ C > 0 $.
\end{proposition}

\end{document}